\documentclass{article}
\usepackage{amsmath, amssymb, comment, amsthm}
\usepackage{latexsym}
\usepackage{nextpage}
\usepackage{graphicx}
\usepackage{multicol}
\usepackage{thmtools}
\usepackage{float}
\usepackage{array, multirow, hhline}

\usepackage{tocloft} 
\usepackage{verbatim} 
\usepackage{setspace} 
\usepackage{indentfirst} 
\usepackage[font=normal,labelfont=normal,labelsep=period,tableposition=top,justification=raggedright]{caption} 
\usepackage{flafter} 
\usepackage{mathrsfs} 
\usepackage{ragged2e} 
\usepackage[compact]{titlesec} 

\usepackage{todonotes}

\makeatletter
\renewcommand\section{\@startsection {section}{1}{\z@}%
{-3.5ex \@plus -1ex \@minus -.2ex}%
{2.3ex \@plus.2ex}%
{\normalfont\centering}}
\makeatother

\makeatletter
\renewcommand\subsection{\@startsection {subsection}{1}{\z@}%
{-3.5ex \@plus -1ex \@minus -.2ex}%
{0.8pt plus 2pt minus 2pt}%
{\normalfont}}
\makeatother

\makeatletter
\renewcommand\subsubsection{\@startsection {subsubsection}{1}{\z@}%
{-3.5ex \@plus -1ex \@minus -.2ex}%
{0.8pt plus 2pt minus 2pt}%
{\normalfont}}
\makeatother

\usepackage{fancyhdr}
\fancyheadoffset[HR]{0.02in}

\usepackage{anysize}
\marginsize{1.7in}{1in}{1.2in}{1.6in}

\usepackage[all]{nowidow} 
\usepackage{array}
\usepackage{tikz}

\newtheorem{lemma}{Lemma}
\newtheorem{theorem}[lemma]{Theorem}
\newtheorem{corollary}[lemma]{Corollary}
\newtheorem{proposition}[lemma]{Proposition}
\newtheorem{defnition}[lemma]{Definition} 

\newenvironment{definition}{\begin{defnition} \em}{\end{defnition}}

\theoremstyle{remark}

\newcommand{\R}{\operatorname{Re}}
\newcommand{\I}{\operatorname{Im}}

\renewcommand{\liminf}{ \underline{\lim} \,}
\renewcommand{\limsup}{\overline{\lim} \,}

\begin{document}

\begin{center}
ON HYPERBOLIC POLYNOMIALS WITH FOUR-TERM RECURRENCE AND LINEAR COEFFICIENTS \\
\bigskip
RICHARD ADAMS
\end{center}

ABSTRACT. For any real numbers $a,\ b$, and $c$, we form the sequence of polynomials $\{P_n(z)\}_{n=0}^\infty$ satisfying the four-term recurrence
\[
P_n(z)+azP_{n-1}(z)+bP_{n-2}(z)+czP_{n-3}(z)=0,\ n\in\mathbb{N},
\]
with the initial conditions $P_0(z)=1$ and $P_{-n}(z)=0$. We find necessary and sufficient conditions on $a,\ b$, and $c$ under which the zeros of $P_n(z)$ are real for all $n$, and provide an explicit real interval on which $\displaystyle\bigcup_{n=0}^\infty\mathcal{Z}(P_n)$ is dense, where $\mathcal{Z}(P_n)$ is the set of zeros of $P_n(z)$.


\begin{center}
\section{INTRODUCTION}
\end{center}
\pagenumbering{arabic}
 Recursively-defined sequences of polynomials have been studied extensively since the 18th century. Of particular interest are orthogonal polynomials, such as Hermite polynomials, Laguerre polynomials, and Jacobi polynomials, which have numerous applications in differential equations, mathematical and numerical analysis, and approximation theory (see \cite{Orthogonal,Orthogonal1}). Common orthogonal polynomials such as the Chebychev polynomials can be defined by a three-term recurrence relation, which in modern times has been generalized significantly (see \cite{Ch,Forgacs,He,Si,Tran1,Tran}). Not much is known, however, about four-term recurrences. For recent work on the four-term reccurence, see \cite{Eg,Goh,Zumba}. Consider the sequence of polynomials $\{P_n(z)\}_{n=0}^\infty$ satisfying the four-term recurrence relation 
\[
P_n(z)+A(z)P_{n-1}(z)+B(z)P_{n-2}(z)+C(z)P_{n-3}(z)=0,
\]
where $P_0(z)=1$, $P_{-n}(z)=0$, and $A(z),\ B(z)$, and $C(z)$ are some linear complex-valued polynomials with real coefficients. With certain initial conditions, one may wish to find where the zeros of $P_n(z)$ lie on the complex plane for each $n\in\mathbb{N}$. Let $\mathcal{Z}(P_n)=\{z\in\mathbb{C}|P_n(z)=0\}$. We say that a polynomial is hyperbolic if all of its zeros are real, and we are interested in finding necessary and sufficient conditions under which $P_n(z)$ is hyperbolic for all $n\in\mathbb{N}$. In \cite{Zumba} the authors characterized the case where $A(z)=a$, $B(z)=b$, and $C(z)=z$. In this paper, we characterize the case where $A(z)=az$,  $B(z)=b$, and $C(z)=cz$ for some nonzero real numbers $a,\ b$, and $c$. Next we present the main result of this paper.
\begin{theorem}\label{main1}
The zeros of the sequence $\{P_n(z)\}_{n=0}^\infty$ satisfying the recurrence relation 
\[
P_n(z)+azP_{n-1}(z)+bP_{n-2}(z)+czP_{n-3}(z)=0,\ a,b,c\in\mathbb{R}\setminus\{0\},
\]
where $P_0(z)=1$ and $P_{-n}(z)=0$,
are real if and only if $b>0$ and $\displaystyle\frac{c}{ab}:=\alpha\leq\frac{1}{9}$, in which case they lie on the interval $(-\lambda,\lambda)$, where 
\[
\lambda:=\frac{4}{\left(\frac{3\alpha+1+\sqrt{9\alpha^2-10\alpha+1}}{-5\alpha+1+\sqrt{9\alpha^2-10\alpha+1}}\right)^{\frac{3}{2}}\left(-5\alpha+1+\sqrt{9\alpha^2-10\alpha+1}\right)}.
\]
Furthermore, the set of zeros of $P_n(z)$ is dense on $(-\lambda,\lambda)$ as $n\to\infty$.
\end{theorem}
This paper is organized as follows. In Section 2, we restate Theorem \ref{main1} using its generating function and define a few functions which will be used to prove the sufficiency condition of Theorem \ref{main1}. In Section 3, we prove the monotonicity of a special function defined in Section 2, which will then be used to make heuristic arguments in Section 4. Finally, we prove the necessary condition in Section 5, and mention an open problem in Section 6.


\section{THE GENERATING FUNCTION AND SOME DEFINITIONS}

We begin this section by restating Theorem \ref{main1} using the generating function for $\{P_n(z)\}_{n=0}^\infty$.
\begin{theorem}\label{main2}
The zeros of the sequence $\{P_n(z)\}_{n=0}^\infty$ generated by 
\[
\sum_{n=0}^\infty P_n(z)t^n=\frac{1}{1+azt+bt^2+czt^3},\quad a,b,c\in\mathbb{R}\setminus\{0\},
\]
are real if and only if $b>0$ and  $\displaystyle\frac{c}{ab}:=\alpha\leq\frac{1}{9}$, in which case they lie on the interval $(-\lambda,\lambda)$
where 
\[
\lambda:=\frac{4}{\left(\frac{3\alpha+1+\sqrt{9\alpha^2-10\alpha+1}}{-5\alpha+1+\sqrt{9\alpha^2-10\alpha+1}}\right)^{\frac{3}{2}}\left(-5\alpha+1+\sqrt{9\alpha^2-10\alpha+1}\right)}.
\]
Furthermore, the set of zeros of $P_n(z)$ is dense on $(-\lambda,\lambda)$ as $n\to\infty$.
\end{theorem}
If $c=0$ and $a$ and $b$ are nonzero, the denominator becomes a quadratic, which was studied in \cite{Tran1}. If $b=0$ and $a$ and $c$ are nonzero, the denominator becomes a three-term cubic, which was also studied in \cite{Tran1}. If $a=0$ and $b$ and $c$ are nonzero, the denominator becomes a different type of three-term cubic, which was studied in \cite{Zumba}. We will show later that if $b<0$ then $P_n(z)$ is not hyperbolic for all large $n$ (c.f. Proposition \ref{b}).
To simplify the proof, we make the substitutions $\displaystyle t\to\frac{t}{\sqrt{b}}$, $\displaystyle\frac{a}{\sqrt{b}}z\to z'$, and $\displaystyle\frac{c}{ab}\to\alpha$, and the theorem can be restated in the following equivalent form.
\begin{theorem}\label{main3}
The zeros of the sequence $\{P_n(z)\}_{n=0}^\infty$ generated by 
\[
\sum_{n=0}^\infty P_n(z)t^n=\frac{1}{1+zt+t^2+\alpha zt^3},\quad \alpha\in\mathbb{R}\setminus\{0\},
\]
are real if and only if $\displaystyle\alpha\leq\frac{1}{9}$, in which case they lie on the interval $(-\lambda,\lambda)$
where 
\[
\lambda:=\frac{4}{\left(\frac{3\alpha+1+\sqrt{9\alpha^2-10\alpha+1}}{-5\alpha+1+\sqrt{9\alpha^2-10\alpha+1}}\right)^{\frac{3}{2}}\left(-5\alpha+1+\sqrt{9\alpha^2-10\alpha+1}\right)}.
\]
Furthermore, the set of zeros of $P_n(z)$ is dense on $(-\lambda,\lambda)$ as $n\to\infty$.
\end{theorem}
To prove this theorem, we must establish several lemmas. We begin by defining a few terms. 
\begin{definition}\label{functions}
For $\theta\in(0,\pi)$, define the functions
\begin{align*}
t_1(\theta)&:=\tau(\theta)e^{-i\theta}, \\
t_2(\theta)&:=\tau(\theta)e^{i\theta}, \\
t_3(\theta)&:=\tau(\theta)\zeta(\theta), \\
z(\theta)&:=\frac{-1}{\alpha\tau^3(\theta)\zeta(\theta)}, \\
\Delta(\theta)&:=16\alpha^2\cos^4\theta-8\alpha^2\cos^2\theta-8\alpha\cos^2\theta+\alpha^2-2\alpha+1,
\end{align*}
where 
\begin{align*}
\tau(\theta)&:=\sqrt{\frac{2\cos\theta+\zeta(\theta)}{\zeta(\theta)}}\qquad \text{ and} \\
\zeta(\theta)&:=\frac{-4\alpha\cos^2\theta-\alpha+1+\sqrt{\Delta(\theta)}}{4\alpha\cos\theta}.
\end{align*}
\end{definition}
The following three lemmas establish the fact that $\tau,\ \zeta$, and $z$ are well-defined.
\begin{lemma}\label{delta}
For $\displaystyle0\neq\alpha\leq\frac{1}{9}$ and $\theta\in(0,\pi)$, $\Delta(\theta)>0$.
\end{lemma}
\begin{proof}
First suppose $\alpha=1/9$. Then we have
$\Delta(\theta)=\frac{16}{81}(\cos^4\theta-5\cos^2\theta+4)$,
which has zeros at $\theta=0$ and $\theta=\pi$, and none in between. Since $\pi/2\in(0,\pi)$ and 
$\Delta\left(\frac{\pi}{2}\right)=\frac{64}{81}>0$,
and since $\Delta(\theta)$ is continuous on $\mathbb{R}$, we have that $\Delta(\theta)>0$ for $\alpha=1/9$. \\
Now suppose $\alpha<1/9$ and $\alpha\neq0$. 
By differentiating $\Delta(\theta)$ with respect to $\theta$, we find that the relative extrema occur at $\theta=0,\ \pi/2,\ \pi$, and $\cos^{-1}(\pm\sqrt{\frac{\alpha+1}{4\alpha}})$.
There are two conditions that must be met for the latter solutions to exist. First, $\frac{\alpha+1}{4\alpha}$ must be greater than or equal to zero, which is not the case for $\alpha\in(-1,0]$. Furthermore, for $\pm\sqrt{\frac{\alpha+1}{4\alpha}}$ to be in the domain of the inverse cosine function, we must have 
$\left|\pm\sqrt{\frac{\alpha+1}{4\alpha}}\right|\leq1$,
or equivalently,
$\alpha\in(-\infty,0)\cup[1/3,\infty)$.
Since we are assuming that $\alpha<1/9$, we deduce that the inverse cosine solutions only exist for $\alpha\leq-1$. Observe that $\Delta\left(\cos^{-1}\left(\pm\sqrt{\frac{\alpha+1}{4\alpha}}\right)\right)=-4\alpha$,
which is positive for all $\alpha\leq-1$.
Additionally,
$\Delta(0)=\Delta(\pi)=9\alpha^2-10\alpha+1=(9\alpha-1)(\alpha-1)$,
which is positive for all $\alpha<1/9$.
We also have
$\Delta(\pi/2)=\alpha^2-2\alpha+1=(\alpha-1)^2$,
which is positive for all $\alpha\neq1$. 
\\
Since all relative extrema on $(0,\pi)$ give positive values of $\Delta(\theta)$, and $\Delta(\theta)$ is continuous on $\mathbb{R}$, we have that $\Delta(\theta)>0$ for $\theta\in(0,\pi)$.
\end{proof}
\begin{lemma}\label{one}
For $\theta\in(0,\pi)$, let $\zeta$ be as in Definition \ref{functions} with $4\alpha\cos\theta\neq0$. Then $|\zeta|>1$.
\end{lemma}
\begin{proof}
When deriving the formula for $\zeta$ we used combinations of $t_1$, $t_2$, and $t_3$ with Vieta's formulas to obtain the function 
\begin{equation}\label{f}
f(\zeta):=2\alpha\cos\theta\zeta^2+(4\alpha\cos^2\theta+\alpha-1)\zeta+2\alpha\cos\theta,
\end{equation}
which has zeros 
\begin{align*}
\zeta_+&:=\zeta=\frac{-4\alpha\cos^2\theta-\alpha+1+\sqrt{\Delta(\theta)}}{4\alpha\cos\theta}\ \text{ and} \\
\zeta_-&:=\frac{-4\alpha\cos^2\theta-\alpha+1-\sqrt{\Delta(\theta)}}{4\alpha\cos\theta}.
\end{align*}
We have
$f(-1)(f(1)=-\Delta(\theta)$,
which is negative for $\theta\in(0,\pi)$ be Lemma \ref{delta}. Hence, by the Intermediate Value Theorem, exactly one of the zeros of $f(\zeta)$ lies outside the interval $[-1,1]$.
If $\alpha<0$, then $-4\alpha\cos^2\theta-\alpha+1>0$, which implies $|\zeta_+|>|\zeta_-|$.
If $0<\alpha\leq1/9$, then
$-4\alpha\cos^2\theta-\alpha+1\geq-5\alpha+1>0$,
and applying the triangle inequality to the numerators, we have
\[
|-4\alpha\cos^2\theta-\alpha+1-\sqrt{\Delta(\theta)}|\leq-4\alpha\cos^2\theta-\alpha+1+\sqrt{\Delta(\theta)},
\]
which again implies that $|\zeta_+|>|\zeta_-|$. Hence we see
that $\zeta_+=\zeta$ must lie outside the interval $[-1,1]$ for all $\theta\in(0,\pi)$, and thus $|\zeta|>1\ \forall\theta\in(0,\pi)$. 
\end{proof}
\begin{lemma}\label{tau}
For $0\neq\alpha\leq1/9$ and $\theta\in(0,\pi)$, $\displaystyle\frac{2\cos\theta+\zeta}{\zeta}>0$.
\end{lemma}
\begin{proof}
Case 1: Suppose $0<\alpha\leq1/9$. Note that 
\[
\frac{2\cos\theta+\zeta}{\zeta}=\frac{4\alpha\cos^2\theta-\alpha+1+\sqrt{\Delta(\theta)}}{-4\alpha\cos^2\theta-\alpha+1+\sqrt{\Delta(\theta)}}.
\]
We show that both the numerator, denoted $N(\theta)$, and denominator, denoted $D(\theta)$, are positive for all $0<\alpha\leq1/9$. Note that the smallest $D(\theta)$ can be is when $\alpha=1/9$ and $\cos^2\theta=1$. In that case, 
$D(\theta)=\frac{4}{9}+\sqrt{\Delta(\theta)}>0$
by Lemma \ref{delta}. Since $4\cos^2\theta-1<0$ on $(\pi/3,2\pi/3)$, the smallest it can be on that interval is when $\alpha=1/9$ and $4\cos^2\theta-1=-1$. In that case,
$N(\theta)=\frac{8}{9}+\sqrt{\Delta(\theta)}>0$
 by Lemma \ref{delta}. Since $4\cos^2\theta-1>0$ on $(0,\pi/3)\cup(2\pi/3,\pi)$, we get $N(\theta)=\alpha(4\cos^2\theta-1)+1+\sqrt{\Delta(\theta)}>0$ by Lemma \ref{delta}. \\
Case 2: Suppose $\alpha\leq0$. We show that
$\displaystyle\frac{2\cos\theta}{\zeta}+1>0$ by showing that $|\zeta|>|2\cos\theta|$ using a similar argument as in Lemma \ref{one}. From (\ref{f}) in Lemma \ref{one}, we have 
$f(-2\cos\theta)f(2\cos\theta)=4\cos^2\theta(8\alpha\cos^2\theta+2\alpha-1)<0$
since $\cos^2\theta>0$ and $\alpha<0$. Hence, by the Intermediate Value Theorem, exactly one of the zeros of $f(\zeta)$ lies outside the interval $[-2\cos\theta,2\cos\theta]$, and since $|\zeta_+|>|\zeta_-|$ by the proof of Lemma \ref{one}, we get that $|\zeta|>|2\cos\theta|$.
\end{proof}
Note that for each $\theta\in(0,\pi)$, $t_1:=t_1(\theta),\ t_2:=t_2(\theta)$, and $t_3:=t_3(\theta)$ are the zeros of the denominator of the generating function, $1+z(\theta)t+t^2+\alpha z(\theta)t^3$, since they satisfy the relations (known as Vieta's formulas)
\begin{align*}
t_1+t_2+t_3&=-\frac{1}{\alpha z}, \\
t_1t_2+t_1t_3+t_2t_3&=\frac{1}{\alpha}, \\
t_1t_2t_3&=-\frac{1}{\alpha z}.
\end{align*}

\section{THE MONOTONICITY OF $z(\theta)$}

In this section, we establish the fact that the function $z(\theta)$ is increasing on $(0,\pi)$.

\begin{lemma}\label{monotone}
The function $z(\theta)$ in Definition \ref{functions} is monotone increasing on $(0,\pi)$. 
\end{lemma}
\begin{proof}
Since $t_1$ is a zero of $1+z(\theta)t+t^2+\alpha z(\theta)t^3$, 
\[
z(\theta)=-\frac{t_1^2+1}{t_1+\alpha t_1^3}.
\]
First suppose that $\alpha<0$. Then we have 
\[
z(\theta)=-\frac{(t_1+i)(t_1-i)}{t_1(\sqrt{|\alpha|}t_1+1)(\sqrt{|\alpha|}t_1-1)}.
\]
Taking the derivative by the product rule and then dividing by $z(\theta)$ implies that 
\[
\frac{dz}{z}=\frac{dt_1}{t_1+i}+\frac{dt_1}{t_1-i}-\frac{dt_1}{t_1}-\frac{\sqrt{|\alpha|}dt_1}{\sqrt{|\alpha|}t_1+1}-\frac{\sqrt{|\alpha|}dt_1}{\sqrt{|\alpha|}t_1-1}.
\]
Note that 
\[
dt_1=d(\tau e^{-i\theta})=-i\tau e^{-i\theta}d\theta+e^{-i\theta}d\tau=t_1\left(\frac{d\tau}{\tau}-id\theta\right),
\]
so if we let 
\[
h_1(t_1):=\frac{t_1}{t_1+i}+\frac{t_1}{t_1-i}-\frac{t_1\sqrt{|\alpha|}}{\sqrt{|\alpha|}t_1+1}-\frac{t_1\sqrt{|\alpha|}}{\sqrt{|\alpha|}t_1-1}-1,
\]
we have 
\[
\frac{dz}{z}=h_1(t_1)\left(\frac{d\tau}{\tau}-id\theta\right), 
\]
or equivalently,
\[
\frac{dz}{zd\theta}=h_1(t_1)\left(\frac{d\tau}{\tau d\theta}-i\right).
\]
Since $\displaystyle\frac{dz}{zd\theta}\in\mathbb{R}$ for each $\theta\in(0,\pi)$, we have 
\[
0=\I\frac{dz}{zd\theta}=\I\left(h_1(t_1)\frac{d\tau}{\tau d\theta}-h_1(t_1)i\right)=-\R h_1(t_1)+\I h_1(t_1)\frac{d\tau}{\tau d\theta},
\]
which implies 
\[
\I h_1(t_1)\frac{d\tau}{\tau d\theta}=\R h_1(t_1)
\]
and 
\begin{equation}\label{dz}
\frac{dz}{zd\theta}=\R h_1(t_1)\frac{d\tau}{\tau d\theta}+\I h_1(t_1).
\end{equation}
If we multiply both sides of Equation (\ref{dz}) by $\I h_1(t_1)$, we have 
\begin{align*}
\I h_1(t_1)\frac{dz}{zd\theta}&=\I h_1(t_1)\frac{d\tau}{\tau d\theta}\R h_1(t_1)+\I h_1(t_1)\I h_1(t_1) \\
&=\R h_1(t_1) \R h_1(t_1)+\I h_1(t_1)\I h_1(t_1) \\
&=(\R h_1(t_1))^2+(\I h_1(t_1))^2 \\
&=|h_1(t_1)|^2,
\end{align*}
which implies
\[
\frac{dz}{d\theta}=\frac{z|h_1(t_1)|^2}{\I h_1(t_1)}. 
\]
Since $|h_1(t_1)|^2>0$, the sign of $\displaystyle\frac{dz}{d\theta}$ depends on $z$ and $\I h_1(t_1)$. Note that
\begin{equation}\label{Im}
\I h_1(t_1)=\I\left(\frac{t_1}{t_1+i}+\frac{t_1}{t_1-i}-\frac{t_1\sqrt{|\alpha|}}{\sqrt{|\alpha|}t_1+1}-\frac{t_1\sqrt{|\alpha|}}{\sqrt{|\alpha|}t_1-1}-1\right). 
\end{equation}
Multiplying the top and bottom of each fraction by the conjugate of their respective denominator makes Equation (\ref{Im}) equivalent to 
\[
\I\left(\frac{|t_1|^2-t_1i}{|t_1+i|^2}+\frac{|t_1|^2+t_1i}{|t_1-i|^2}-\frac{|\alpha||t_1|^2+t_1\sqrt{|\alpha|}}{|\sqrt{|\alpha|}t_1+1|^2}-\frac{|\alpha||t_1|^2-t_1\sqrt{|\alpha|}}{|\sqrt{|\alpha|}t_1-1|^2}\right).
\]
Since $|t_1|^2\in\mathbb{R}$ and $|\alpha|\in\mathbb{R}$, $\I h_1(t_1)$ is equivalent to
\[
\I\left(-\frac{t_1i}{|t_1+i|^2}+\frac{t_1i}{|t_1-i|^2}-\frac{t_1\sqrt{|\alpha|}}{|\sqrt{|\alpha|}t_1+1|^2}+\frac{t_1\sqrt{|\alpha|}}{|\sqrt{|\alpha|}t_1-1|^2}\right).
\]
We next substitute the expression in Definition \ref{functions} for $t_1$ and find 
\begin{align*}
    \I h_1(t_1)=&-\frac{\tau\cos\theta}{\tau^2-2\tau\sin\theta+1}+\frac{\tau\cos\theta}{\tau^2+2\tau\sin\theta+1} \\
    &+\frac{\sqrt{|\alpha|}\tau\sin\theta}{|\alpha|\tau^2+2\sqrt{|\alpha|}\tau\cos\theta+1}-\frac{\sqrt{|\alpha|}\tau\sin\theta}{|\alpha|\tau^2-2\sqrt{|\alpha|}\tau\cos\theta+1}.
\end{align*}
When we combine, the common denominator will be a product of squares of modulus, and will thus be a positive value. Hence we can consider just the numerator, which becomes 
\begin{align*}
  &-4\alpha^2\tau^6\sin\theta\cos\theta-16|\alpha|\tau^4\sin\theta\cos\theta+16|\alpha|\tau^4\sin\theta\cos^3\theta \\
 &-4\tau^2\sin\theta\cos\theta-4|\alpha|\tau^6\sin\theta\cos\theta+16|\alpha|\tau^4\sin^3\theta\cos\theta-4|\alpha|\tau^2\sin\theta\cos\theta. 
\end{align*}
Applying the trigonometric identities $2\sin\theta\cos\theta=\sin2\theta$, and $\sin^2\theta+\cos^2\theta=1$, we obtain
\[
    \I h_1(t_1)=-2\tau^2(\sin2\theta+|\alpha|\sin2\theta+|\alpha|\tau^4\sin2\theta+\alpha^2\tau^4\sin2\theta).
\]
Since $-2\tau^2<0$, we analyze 
\[
\sin2\theta+|\alpha|\sin2\theta+|\alpha|\tau^4\sin2\theta+\alpha^2\tau^4\sin2\theta=\sin2\theta(1+|\alpha|+|\alpha|\tau^4+\alpha^2\tau^4)
\]
on the intervals $(0,\pi/2)$ and $(\pi/2,\pi)$. \\
On $(0,\pi/2)$, $z<0$ and $\sin2\theta>0$, and since $1+|\alpha|+|\alpha|\tau^4+\alpha^2\tau^4>0$, we have that $\I h_1(t_1)<0$, which implies that $\displaystyle\frac{dz}{d\theta}>0$ on $(0,\pi/2)$. \\
On $(\pi/2,\pi)$, $z>0$ and $\sin2\theta<0$, so we get that $\I h_1(t_1)>0$, which implies that $\displaystyle\frac{dz}{d\theta}>0$ on $(\pi/2,\pi)$. Thus $\displaystyle\frac{dz}{d\theta}>0$ on $(0,\pi)$ for $\alpha<0$. \\
Now suppose $0<\alpha\leq1/9$. Then 
\[
z=-\frac{(t_1+i)(t_1-i)}{t_1(\sqrt{\alpha}t_1+i)(\sqrt{\alpha}t_1-i)},
\]
and taking the derivative by the product rule and dividing by $z$ implies that 
\[
\frac{dz}{z}=\frac{dt_1}{t_1+i}+\frac{dt_1}{t_1-i}-\frac{dt_1}{t_1}-\frac{\sqrt{\alpha}dt_1}{\sqrt{\alpha}t_1+i}-\frac{\sqrt{\alpha}dt_1}{\sqrt{\alpha}t_1-i}.
\]
Using a similar process as above, we obtain
\[
\frac{dz}{d\theta}=\frac{z|h_2(t_1)|^2}{\I h_2(t_1)}
\]
where 
\[
h_2(t_1):=\frac{t_1}{t_1+i}+\frac{t_1}{t_1-i}-\frac{t_1\sqrt{\alpha}}{\sqrt{\alpha}t_1+i}-\frac{t_1\sqrt{\alpha}}{\sqrt{\alpha}t_1-i}-1.
\]
Since $|h_2(t_1)|^2>0$, the sign of $\displaystyle\frac{dz}{d\theta}$ again depends on $z$ and $\I h_2(t_1)$. We evaluate $\I h_2(t_1)$ in the same way. We have
\begin{equation}\label{Im3}
\I h_2(t_1)=\I\left(\frac{t_1}{t_1+i}+\frac{t_1}{t_1-i}-\frac{t_1\sqrt{\alpha}}{\sqrt{\alpha}t_1+i}-\frac{t_1\sqrt{\alpha}}{\sqrt{\alpha}t_1-i}-1\right).
\end{equation}
Multiplying the top and bottom of each fraction by the conjugate of their respective denominator makes Equation (\ref{Im3}) equivalent to
\[
\I\left(\frac{|t_1|^2-t_1i}{|t_1+i|^2}+\frac{|t_1|^2+t_1i}{|t_1-i|^2}+\frac{t_1\sqrt{\alpha}i-\alpha|t_1|^2}{|\sqrt{\alpha}t_1+i|^2}+\frac{-t_1\sqrt{\alpha}i-\alpha|t_1|^2}{|\sqrt{\alpha}t_1-i|^2}\right). 
\]
Since $|t_1|^2\in\mathbb{R}$ and $\sqrt{\alpha}\in\mathbb{R}$, $\I h_2(t_1)$ is equivalent to
\[
\I\left(-\frac{t_1i}{|t_1+i|^2}+\frac{t_1i}{|t_1-i|^2}+\frac{t_1\sqrt{\alpha}i}{|\sqrt{\alpha}t_1+i|^2}-\frac{t_1\sqrt{\alpha}i}{|\sqrt{\alpha}t_1-i|^2}\right).
\]
We next substitute the expression in Definition \ref{functions} for $t_1$ and find 
\begin{align*}
    \I h_2(t_1)=&-\frac{\tau\cos\theta}{\tau^2-2\tau\sin\theta+1}+\frac{\tau\cos\theta}{\tau^2+2\tau\sin\theta+1} \\
    &+\frac{\sqrt{\alpha}\tau\cos\theta}{\alpha\tau^2-2\sqrt{\alpha}\tau\sin\theta+1}-\frac{\sqrt{\alpha}\cos\theta}{\alpha\tau^2+2\sqrt{\alpha}\tau\sin\theta+1}.
\end{align*}
As before, when we combine, the common denominator will be a product of squares of modulus, and will thus be a positive value. Hence we can can consider just the numerator, which becomes 
\[
    -2\alpha^2\tau^6\sin\theta\cos\theta-4\tau^2\sin\theta\cos\theta+4\alpha\tau^6\sin\theta\cos\theta+4\alpha\tau^2\sin\theta\cos\theta.
\]
Applying the trigonometric identity $2\sin\theta\cos\theta=\sin2\theta$, we obtain
\[
    \I h_2(t_1)=-2\tau^2(\sin2\theta-\alpha\sin2\theta-\alpha\tau^4\sin2\theta+\alpha^2\tau^4\sin2\theta).
\]
Since $-2\tau^2<0$, we analyze 
\[
\sin2\theta-\alpha\sin2\theta-\alpha\tau^4\sin2\theta+\alpha^2\tau^4\sin2\theta=\sin2\theta(1-\alpha-\alpha\tau^4+\alpha^2\tau^4)
\]
on $(0,\pi/2)$ and $(\pi/2,\pi)$. \\
Suppose first that $\theta\in(0,\pi/2)$. Then $z<0$ and $\sin2\theta>0$. Note that 
\[
1-\alpha-\alpha\tau^4+\alpha^2\tau^4=(1-\alpha)(1-\alpha\tau^4).
\]
Since $1-\alpha>0$ for $0<\alpha\leq1/9$, we must show that $1-\alpha\tau^4>0$. Note that $1-\alpha\tau^4$ is smallest when $\alpha=1/9$. Note that $\displaystyle\tau^4=\left(\frac{2\cos\theta}{\zeta}+1\right)^2$, and since $|\tau|>1$ by Lemma \ref{one}, we have 
\[
0<\frac{2\cos\theta}{\zeta}+1<3\implies\tau^4<9.
\]
Hence $1-\alpha\tau^4>0$, and thus $\displaystyle\frac{dz}{d\theta}>0$ on $(0,\pi/2)$. \\
Now suppose $\theta\in(\pi/2,\pi)$. Then $z>0$ and $\sin2\theta<0$. With the same argument as above, we get that $1-\alpha\tau^4>0$, implying that $\displaystyle\frac{dz}{d\theta}>0$ on $(\pi/2,\pi)$. \\
Hence $\displaystyle\frac{dz}{d\theta}>0$ on $(0,\pi)$, and thus $z$ is increasing on $(0,\pi)$. 
\end{proof}
\begin{corollary}
For $\theta\in(0,\pi)$, $z\in(-\lambda,\lambda)$ where
\[
\lambda:=\frac{4}{\left(\frac{3\alpha+1+\sqrt{9\alpha^2-10\alpha+1}}{-5\alpha+1+\sqrt{9\alpha^2-10\alpha+1}}\right)^{\frac{3}{2}}\left(-5\alpha+1+\sqrt{9\alpha^2-10\alpha+1}\right)}.
\]
\end{corollary}
\begin{proof}
By Lemma \ref{monotone}, $z$ is increasing on $(0,\pi)$, so it suffices to check $\displaystyle\lim_{\theta\to0^+}z(\theta)$ and $\displaystyle\lim_{\theta\to\pi^-}z(\theta)$. It is easy to see that $\displaystyle\lim_{\theta\to0^+}z(\theta)=-\lambda$ and $\displaystyle\lim_{\theta\to\pi^-}z(\theta)=\lambda$.
\end{proof}


\section{HEURISTIC ARGUMENTS AND THE VERTICAL ASYMPTOTE}

In this section, we prove the sufficiency part of Theorem \ref{main1}, being careful to consider all the cases of the discontinuity of the $\zeta$ function at $\theta=\pi/2$. In the following lemma, we define a function which will be used to prove the hyperbolicy of $P_n(z)$.
\begin{lemma}\label{lemma}
For any $\theta\in(0,\pi)$, $P_n(z(\theta))=0$ if and only if $g_n(\theta)=0$, where 
\begin{equation}\label{g}
    g_n(\theta):=\frac{(\zeta-\cos\theta)\sin(n+1)\theta}{\sin\theta}-\cos(n+1)\theta+\frac{1}{\zeta^{n+1}}.
\end{equation}
\end{lemma}
\begin{proof}
This follows directly from (2.4), (2.5), and (2.6) in \cite{Zumba}.
\end{proof}

By Lemma \ref{lemma}, $\theta$ is a zero of $P_n(z(\theta))$ if and only if $g_n(\theta)=0$.
Note that when $\cos(n+1)\theta=\pm1$, $\sin(n+1)\theta=0$.
By Lemma \ref{one}, we know that $|\zeta|>1$, so we see that when $\cos(n+1)\theta=1,\ g_n<0$, and when $\cos(n+1)\theta=-1,\ g_n>0$.
Hence, by the Intermediate Value Theorem, the function $g_n$ has at least one zero on each subinterval whose endpoints are the solutions of $\cos(n+1)\theta=\pm1$. 
However, $g_n$ has a vertical asymptote in the subinterval containing $\theta=\pi/2$ because $\zeta$ has a vertical asymptote at $\theta=\pi/2$.
\\
Note that on $(0,\pi)$, $\cos(n+1)\theta=\pm1$, which implies $\theta=\frac{k\pi}{n+1},\ k\in\{1,2,3,\dots,n\}$,
which means there are $n-1$ subintervals on $(0,\pi)$, excluding the endpoints at $0$ and $\pi$, and as stated above each subinterval besides the one containing the vertical asymptote has at least one zero of $g_n$. So we have found $n-2$ zeros of $g_n$ on $(0,\pi)$. An inductive argument can show that the degree of $P_n(z)$ is $n$, so we must find the missing zeros. Since there is a vertical asymptote at $\theta=\pi/2$, it is possible that the subinterval containing $\pi/2$ has multiple zeros or no zeros, with the latter situation meaning there are two missing zeros. The following lemmas settle these cases. We will need to consider the cases where $\alpha<0$ and $0<\alpha\leq1/9$ separately. Note that the interval $(0,\pi)$ is partitioned into $n+1$ subintervals with endpoints
\[
0,\frac{\pi}{n+1},\frac{2\pi}{n+1},\frac{3\pi}{n+1},\dots,\frac{n\pi}{n+1},\pi.
\]
\begin{lemma}\label{first}
For $\alpha<0$, the function $g_n$ defined in Equation (\ref{g}) has at least one zero on each of the intervals 
\[
\left(0,\frac{\pi}{n+1}\right)\ \text{ and }\ \left(\frac{n\pi}{n+1},\pi\right).
\]
\end{lemma}
\begin{proof}
Recall that 
\[
g_n=\frac{(\zeta-\cos\theta)\sin(n+1)\theta}{\sin\theta}-\cos(n+1)\theta+\frac{1}{\zeta^{n+1}},
\]
and the term $-\cos(n+1)\theta$ determines the sign of $g_n(\theta)$ at each interval endpoint. \\
For $\displaystyle\theta=\frac{\pi}{n+1}$, we have 
\[
-\cos(n+1)\theta=-\cos\pi=-(-1)=1,
\]
which implies that $g_n>0$ at $\displaystyle\frac{\pi}{n+1}\ \forall n\in\mathbb{N}$. \\
Note that 
\[
\lim_{\theta\to0^+}(\zeta-\cos\theta)<0,
\]
and 
\[
\lim_{\theta\to0^+}\frac{\sin(n+1)\theta}{\sin\theta}=n+1.
\]
Furthermore, since $\cos(0)=1$ and $\displaystyle\left|\frac{1}{\zeta^{n+1}}\right|<1$, we have
\[
\lim_{\theta\to0^+}g_n<0\ \forall n\in\mathbb{N}.
\]
Since $g_n$ is continuous on $\displaystyle\left(0,\frac{\pi}{n+1}\right)$ and it changes sign from $0$ to $\displaystyle\frac{\pi}{n+1}$, we have by the Intermediate Value Theorem that $g_n$ has at least one zero on $\displaystyle\left(0,\frac{\pi}{n+1}\right)$. \\
For the interval $\displaystyle\left(\frac{n\pi}{n+1},\pi\right)$, we must consider the two cases when $n$ is even and when $n$ is odd. \\
First, suppose $n$ is even. So $n=2\ell$ for some $\ell\in\mathbb{N}$. \\
For $\displaystyle\theta=\frac{n\pi}{n+1}$, we have 
\[
-\cos(n+1)\theta=-\cos(2\ell\pi)=-1,
\]
which implies $g_n<0$ at $\displaystyle\frac{n\pi}{n+1}\  \forall n\in2\mathbb{N}$. \\
Note that 
\[
\lim_{\theta\to\pi^-}(\zeta-\cos\theta)>0,
\]
and
\[
\lim_{\theta\to\pi^-}\frac{\sin(n+1)\theta}{\sin\theta}=\lim_{\theta\to\pi^-}\frac{\sin(2\ell\theta+\theta)}{\sin\theta}=\lim_{\theta\to\pi^-}\frac{(2\ell+1)\cos(2\ell\theta+\theta)}{\cos\theta}=2\ell+1.
\]
Furthermore, since 
\[
-\cos(2\ell+1)\pi=-\cos(2\ell\pi+\pi)=-(-1)=1,
\]
and 
\[
\left|\frac{1}{\zeta^{n+1}}\right|<1,
\]
we ascertain that
\[
\lim_{\theta\to\pi^-}g_n>0\ \forall n\in2\mathbb{N}.
\]
Since $g_n$ is continuous on $\displaystyle\left(\frac{n\pi}{n+1},\pi\right)$ and it changes sign from $\displaystyle\frac{n\pi}{n+1}$ to $\pi$, the Intermediate Value Theorem implies that $g_n$ has at least one zero on $\displaystyle\left(\frac{n\pi}{n+1},\pi\right)$ if $n$ is even. \\
Now suppose that $n$ is odd and write $n=2\ell+1$ for some $\ell\in\mathbb{N}$. \\
For $\displaystyle\theta=\frac{n\pi}{n+1}$, we have 
\[
-\cos(n+1)\theta=-\cos(2\ell\pi+\pi)=-(-1)=1,
\]
which implies that $g_n>0$ at $\displaystyle\frac{n\pi}{n+1}=\frac{(2\ell +1)\pi}{2\ell+2}\ \forall\ell\in\mathbb{N}$. \\
As before, 
\[
\lim_{\theta\to\pi^-}(\zeta-\cos\theta)>0,
\]
but 
\begin{align*}
\lim_{\theta\to\pi^-}\frac{\sin(n+1)\theta}{\sin\theta}&=\lim_{\theta\to\pi^-}\frac{\sin(2\ell\theta+2\theta)}{\sin\theta} \\
&=\lim_{\theta\to\pi^-}\frac{(2\ell+2)\cos(2\ell\theta+2\theta)}{\cos\theta} \\
&=-(2\ell+2).
\end{align*}
Furthermore, since 
\[
-\cos(2\ell+2)\pi=-1,
\]
and 
\[
\left|\frac{1}{\zeta^{n+1}}\right|<1,
\]
we have 
\[
\lim_{\theta\to\pi^-}g_n<0\ \forall\ell\in\mathbb{N}.
\]
Since $g_n$ is continuous on $\displaystyle\left(\frac{n\pi}{n+1},\pi\right)$ and it changes sign from $\displaystyle\frac{n\pi}{n+1}$ to $\pi$, the Intermediate Value Theorem implies that $g_n$ has at least one zero on $\displaystyle\left(\frac{n\pi}{n+1},\pi\right)$ for $n$ odd.
\end{proof}
In the next lemma, we will use the fact that for $0<\alpha\leq1/9$,
\begin{equation}\label{limit1}
    \lim_{\theta\to\frac{\pi}{2}^-}(\zeta-\cos\theta)=+\infty,
\end{equation}
\begin{equation}\label{limit2}
    \lim_{\theta\to\frac{\pi}{2}^+}(\zeta-\cos\theta)=-\infty,
\end{equation}
and the terms $-\cos(n+1)\theta$ and $\displaystyle\frac{1}{\zeta^{n+1}}$ are finite at $\theta=\pi/2$.
Note that if $n$ is even, then the vertical asymptote at $\displaystyle\theta=\frac{\pi}{2}$ lies in the center of the central subinterval
\[
\left(\frac{n\pi}{2n+2},\frac{(n+2)\pi}{2n+2}\right).
\]
\begin{lemma}\label{second}
For $0<\alpha\leq1/9$ and $n$ even, $g_n$ has at least one zero on the intervals
\[
\left(\frac{n\pi}{2n+2},\frac{\pi}{2}\right)\ \text{ and }\ \left(\frac{\pi}{2},\frac{(n+2)\pi}{2n+2}\right).
\]
\end{lemma}
\begin{proof}
We must consider the case where $n$ is a multiple of $4$ and when $n$ is not a multiple of $4$ separately. \\
First, suppose $n$ is a multiple of $4$. Then $n=4\ell$ for some $\ell\in\mathbb{N}$. So
\begin{equation}\label{limit3}
\lim_{\theta\to\frac{\pi}{2}^\pm}\frac{\sin(n+1)\theta}{\sin\theta}=\sin(2\ell\theta+\pi/2)=1.
\end{equation}
Hence Equations (\ref{limit1}) and (\ref{limit2}) together with Equation (\ref{limit3}) imply that
\begin{equation}\label{limit4}
    \lim_{\theta\to\frac{\pi}{2}^-}g_n(\theta)=+\infty
\end{equation}
and 
\begin{equation}\label{limit5}
    \lim_{\theta\to\frac{\pi}{2}^+}g_n(\theta)=-\infty.
\end{equation}
Next we determine the sign of $g_n$ at the endpoints $\displaystyle\frac{n\pi}{2n+2}$ and $\displaystyle\frac{(n+2)\pi}{2n+2}$ by finding the sign of $-\cos(n+1)\theta$. Observe that for $n=4\ell$ and $\displaystyle\theta=\frac{n\pi}{2n+2}$,
\begin{equation}\label{limit6}
    -\cos(n+1)\theta=-\cos(2\ell\theta)=-1.
\end{equation}
For $n=4\ell$ and $\displaystyle\theta=\frac{(n+2)\pi}{2n+2}$,
\begin{equation}\label{limit7}
    -\cos(n+1)\theta=-\cos(2\ell\pi+\pi)=1.
\end{equation}
Equations (\ref{limit4}) and (\ref{limit6}) imply that $g_n$ changes sign on the interval $\displaystyle\left(\frac{n\pi}{2n+2},\frac{\pi}{2}\right)$, and Equations (\ref{limit5}) and (\ref{limit7}) imply that $g_n$ changes sign on the interval $\displaystyle\left(\frac{\pi}{2},\frac{(n+2)\pi}{2n+2}\right)$. Since $g_n$ is continuous on those intervals, we have by the Intermediate Value Theorem that $g_n$ has at least one zero on each of those intervals for $n=4\ell$. \\
Now suppose $n$ is not a multiple of $4$. So $n=4\ell-2$ for some $\ell\in\mathbb{N}$. So
\begin{equation}\label{limit8}
    \lim_{\theta\to\frac{\pi}{2}^\pm}\frac{\sin(n+1)\theta}{\sin\theta}=\sin(2\ell\pi-\pi/2)=-1.
\end{equation}
Hence Equations (\ref{limit1}) and (\ref{limit2}) together with Equation (\ref{limit8}) imply that 
\begin{equation}\label{limit9}
    \lim_{\theta\to\frac{\pi}{2}^-}g_n(\theta)=-\infty
\end{equation}
and 
\begin{equation}\label{limit10}
    \lim_{\theta\to\frac{\pi}{2}^+}g_n(\theta)=+\infty.
\end{equation}
Again, we determine the sign of $g_n$ at the endpoints by finding the sign of $-\cos(n+1)\theta$. Observe that for $n=4\ell-2$ and $\displaystyle\theta=\frac{n\pi}{2n+2}$, 
\begin{equation}\label{limit11}
    -\cos(n+1)\theta)=-\cos(2\ell\pi-\pi)=1.
\end{equation}
For $n=4\ell-2$ and $\displaystyle\theta=\frac{(n+2)\pi)}{2n+2}$,
\begin{equation}\label{limit12}
    -\cos(n+1)\theta=-\cos(2\ell\pi)=-1.
\end{equation}
Equations (\ref{limit9}) and (\ref{limit11}) imply that $g_n$ changes sign on the interval $\displaystyle\left(\frac{n\pi}{2n+2},\frac{\pi}{2}\right)$, and Equations (\ref{limit10}) and (\ref{limit12}) imply that $g_n$ changes sign on the interval $\displaystyle\left(\frac{\pi}{2},\frac{(n+2)\pi}{2n+2}\right)$. Since $g_n$ is continuous on those intervals, we again have by the Intermediate Value Theorem that $g_n$ has at least one zero on each of those intervals for $n=4\ell-2$.
\end{proof}
Note that if $n$ is odd, then the possible discontinuity at $\displaystyle\theta=\frac{\pi}{2}$ occurs directly on the central subinterval endpoint.
\begin{lemma}\label{third}
For $0<\alpha\leq1/9$ and $n$ odd, $g_n$ has at least one zero on the intervals
\[
\left(\frac{(n-1)\pi}{2n+2},\frac{\pi}{2}\right)\ \text{ and }\ \left(\frac{\pi}{2},\frac{(n+3)\pi}{2n+2}\right).
\]
\end{lemma}
\begin{proof}
Let $n$ be an odd natural number. We must consider the case where $n=4\ell-1$ and $n=4\ell+1,\ \ell\in\mathbb{N}$, separately. \\
First, suppose $n=4\ell-1$. Then since $0<\alpha\leq1/9$, 
\begin{equation}\label{limit13}
   \lim_{\theta\to\frac{\pi}{2}}\frac{(\zeta-\cos\theta)\sin(n+1)\theta}{\sin\theta}=\frac{1-\alpha}{\alpha}(-2\ell).
\end{equation}
Since $|\cos(n+1)\theta|\leq1$ and $\displaystyle\left|\frac{1}{\zeta}\right|<1$, Equation (\ref{limit13}) implies that 
\begin{equation}\label{limit15}
    \lim_{\theta\to\frac{\pi}{2}}g_n(\theta)<0.
\end{equation}
Next we determine the sign of $g_n$ at the endpoints $\displaystyle\frac{(n-1)\pi}{2n+2}$ and $\displaystyle\frac{(n+3)\pi}{2n+2}$ by finding the sign of $-\cos(n+1)\theta$. Observe that for $n=4\ell-1$ and $\displaystyle\theta=\frac{(n-1)\pi}{2n+2}$,
\begin{equation}\label{limit17}
    -\cos(n+1)\theta=-\cos(2\ell\pi-\pi)=1.
\end{equation}
For $n=4\ell-1$ and $\displaystyle\theta=\frac{(n+3)\pi}{2n+2}$,
\begin{equation}\label{limit18}
    -\cos(n+1)\theta=-\cos(2\ell\pi+\pi)=1.
\end{equation}
Equations (\ref{limit15}) and (\ref{limit17}) imply that $g_n$ changes sign on the interval $\displaystyle\left(\frac{(n-1)\pi}{2n+2},\frac{\pi}{2}\right)$, and Equations (\ref{limit15}) and (\ref{limit18}) imply that $g_n$ changes sign on the interval $\displaystyle\left(\frac{\pi}{2},\frac{(n+3)\pi}{2n+2}\right)$. Since $g_n$ is continuous on those intervals, we have by the Intermediate Value Theorem that $g_n$ has at least one zero on each of those intervals for $n=4\ell-1$. \\
Now suppose $n=4\ell+1$ for some $\ell\in\mathbb{N}$. Then since $0<\alpha\leq1/9$,
\begin{equation}\label{limit19}
    \lim_{\theta\to\frac{\pi}{2}}\frac{(\zeta-\cos\theta)\sin(n+1)\theta}{\sin\theta}=\frac{1-\alpha}{\alpha}(2\ell+1).
\end{equation}
Since $|\cos(n+1)\theta|\leq1$ and $\displaystyle\left|\frac{1}{\zeta}\right|<1$, Equation (\ref{limit19}) implies that
\begin{equation}\label{limit21}
    \lim_{\theta\to\frac{\pi}{2}}g_n(\theta)>0.
\end{equation}
We again determine the sign of $g_n$ at the endpoints $\displaystyle\frac{(n-1)\pi}{2n+2}$ and $\displaystyle\frac{(n+3)\pi}{2n+2}$ by finding the sign of $-\cos(n+1)\theta$. Observe that for $n=4\ell+1$ and $\displaystyle\theta=\frac{(n-1)\pi}{2n+2}$,
\begin{equation}\label{limit23}
    -\cos(n+1)\theta=-\cos(2\ell\pi)=-1.
\end{equation}
For $n=4\ell+1$ and $\displaystyle\theta=\frac{(n+3)\pi}{2n+2}$, 
\begin{equation}\label{limit24}
    -\cos(n+1)\theta=-\cos(2\ell\pi+2\pi)=-1.
\end{equation}
Equations (\ref{limit21}) and (\ref{limit23}) imply that $g_n$ changes sign on the interval $\displaystyle\left(\frac{(n-1)\pi}{2n+2},\frac{\pi}{2}\right)$, and Equations (\ref{limit21}) and (\ref{limit24}) imply that $g_n$ changes sign on the interval $\displaystyle\left(\frac{\pi}{2},\frac{(n+3)\pi}{2n+2}\right)$. Since $g_n$ is continuous on those intervals, we again have by the Intermediate Value Theorem that $g_n$ has at least one zero on each of those intervals for $n=4\ell+1$.
\end{proof}
By Lemmas \ref{first}, \ref{second}, and \ref{third}, we see that $g_n$ has a total of $n$ zeros on the interval $(0,\pi)$ when $n$ is even, and a total of $n-1$ zeros on the interval $(0,\pi)$ when $n$ is odd. Since the degree of $P_n(z)$ is $n$ for all $n\in\mathbb{N}$, the Fundamental Theorem of Algebra tells us that there must be exactly one more zero when $n$ is odd. Since complex zeros always occur in conjugate pairs, the single missing zero must be real.
\begin{lemma}\label{odd}
If $n$ is odd, then $z=0$ is a zero of $P_n(z)$.
\end{lemma}
\begin{proof}
Recall that the generating relation for $\{P_n(z)\}_{n=0}^\infty$ is 
\[
\sum_{n=0}^\infty P_n(z)t^n=\frac{1}{1+zt+t^2+\alpha zt^3}.
\]
Then the generating relation for $z=0$ becomes
\[
\sum_{n=0}^\infty P_n(0)t^n=\frac{1}{1+t^2}.
\]
Expanding both sides gives
\[
P_0(0)+P_1(0)t+P_2(0)t^2+P_3(0)t^3+P_4(0)t^4+\dots=1-t^2+t^4-t^6+t^8-\dots
\]
Equating coefficients yields
\[
P_n(0)=
\begin{cases}
&0\ \ \ \text{ if $n$ is odd} \\
&\pm1\ \text{ if $n$ is even}.
\end{cases}
\]
Hence $P_n(0)=0$ for all odd $n$.
\end{proof}
By Lemma \ref{lemma}, we conclude that since $g_n$ has $n$ zeros on $(0,\pi)$ for all $n\in\mathbb{N}$, $P_n(z)$ has $n$ zeros on the interval $(-\lambda,\lambda)$ via the monotone map $z(\theta):(0,\pi)\to(-\lambda,\lambda)$. Since the degree of $P_n(z)$ is $n$ for each $n$, it follows by the Fundamental Theorem of Algebra that all of the zeros of $P_n(z)$ are real and lie on the interval $(-\lambda,\lambda)$.
\begin{lemma}
The set of zeros of $P_n(z)$ is dense on the open interval $ (-\lambda,\lambda)$
as $n\to\infty$. 
\end{lemma}
\begin{proof}
Since the solutions of $\cos(n+1)\theta=\pm1$ are dense on $(0,\pi)$ as $n\to\infty$, it is clear that $\displaystyle\bigcup_{n=0}^\infty\mathcal{Z}(g_n)$
is dense on $(0,\pi)$. Since $z:(0,\pi)\to(-\lambda,\lambda)$ is a continuous bijective function, it follows that $\displaystyle\bigcup_{n=0}^\infty\mathcal{Z}(P_n)$ 
is dense on $(-\lambda,\lambda)$.
\end{proof}
This concludes the sufficiency part of the proof, where $\alpha\leq1/9$ implies that $P_n(z)$ is hyperbolic for all $n$. In the next section, we focus on the necessary part of the proof, where $\alpha>1/9$ or $b<0$ implies that $P_n(z)$ is not hyperbolic for all large $n$.

\begin{center}
\section{NECESSARY CONDITION FOR THE REALITY OF ZEROS}
\end{center}
To prove that $P_n(z)$ is not hyperbolic for large $n$ when $b<0$ or $\alpha>1/9$, we will employ the Implicit Function Theorem and a theorem by Sokal in \cite{Sokal}, but we must first define a few terms to make sense of the statement of Sokal's theorem. 
\begin{definition}\label{akbk}
Let 
\[
f_n(z)=\sum_{k=1}^m\alpha_k(z)\beta_k(z)^n,\ n\in\mathbb{N},
\]
be a function where $\alpha_k(z)$ and $\beta_k(z)$ are analytic in a domain $D$. We say that an index $k$ is dominant at $z$ if $|\beta_k(z)|\geq|\beta_l(z)|$ for all $1\leq l\leq m$. 
\end{definition}
\begin{definition}\label{lim}
Let 
\[
D_k:=\{z\in D|\ k\text{ is dominant at }z\}.
\]
We say that $\liminf\mathcal{Z}(f_n)$ is the set of all $z\in D$ such that every neighborhood $U$ of $z$ has a nonempty intersection with all but finitely many of the sets $\mathcal{Z}(f_n)$, and $\limsup\mathcal{Z}(f_n)$ is the set of all $z\in D$ such that every neighborhood $U$ of $z$ has a nonempty intersection with infinitely many of the sets $\mathcal{Z}(f_n)$. 
\end{definition}
In other words, Definition \ref{lim} says that $z\in\liminf\mathcal{Z}(f_n)$ if for any neighborhood $U$ of $z$, $f_n$ has a zero in $U$ for all $n$ large enough, and $z\in\limsup\mathcal{Z}(f_n)$ if for any neighborhood $U$ of $z$, infinitely many $f_n$ have a zero in $U$.
\begin{theorem}\label{sokal} (Sokal \cite{Sokal}) Let $D$ be a domain in $\mathbb{C}$ and let $\alpha_1,\dots,\alpha_m,\beta_1,\dots,\beta_m$ $(m\geq2)$ be analytic functions on $D$, none of which is identically zero. Let us further assume a ``no degenerate dominance" condition: there do not exist indices $k\neq k'$ such that $\beta_k\equiv\omega\beta_{k'}$ for some constant $\omega$ with $|\omega|=1$ and such that $D_k\ (=D_{k'})$ has nonempty interior. For each integer $n\geq0$, define $f_n$ by 
\begin{equation}\label{fn}
f_n(z)=\sum_{k=1}^m\alpha_k(z)\beta_k(z)^n.
\end{equation}
Then $\liminf\mathcal{Z}(f_n)=\limsup\mathcal{Z}(f_n)$, and a point $z^*$ lies in this set if and only if either 
\begin{enumerate}
\item there is a unique dominant index $k$ at $z^*$, and $\alpha_k(z^*)=0$, or 
\item there are two or more dominant indices at $z^*$.
\end{enumerate}
\end{theorem}

We will also apply the Implicit Function Theorem, which we state as the following theorem.
\begin{theorem}(Implicit Function Theorem \cite{Implicit})\label{implicit}
Let $f_j(w,z),\ j=1,\dots,m$, be analytic functions of $(w,z)=(w_1,\dots,w_m,z_1,\dots,z_n)$ in a neighborhood of a point $(w^*,z^*)$ in $\mathbb{C}^m\times\mathbb{C}^n$, and assume that $f_j(w^*,z^*)=0,\ j=1,\dots,m$, and that
\[
\det\left(\frac{\partial f_j}{\partial w_k}\right)_{j,k=1}^m\neq0\ \text{ at }\ (w^*,z^*).
\]
Then the equations $f_j(w,z)=0,\ j=1,\dots,m$, have a uniquely determined analytic solution $w(z)$ in a neighborhood of $z^*$ such that $w(z^*)=w^*$.
\end{theorem}
We will later find a $z^*\in\mathbb{C}\setminus\mathbb{R}$ such that the zeros $t_1^*,\ t_2^*$, and $t_3^*$ of $1+z^*t+t^2+\alpha z^*t^3$ satisfy 
\[
|t_1^*|=|t_2^*|\leq|t_3^*|,
\]
where $t_1^*,\ t_2^*$, and $t_3^*$ are distinct and nonzero. Assuming that we can find such a $z^*$, we now claim that $z^*\in\liminf\mathcal{Z}(P_n)$, which implies that $P_n(z)$ is not hyperbolic for all large $n$.

First of all, Theorem \ref{implicit} implies that for a fixed $z^*\in\mathbb{C}\setminus\mathbb{R}$, there exist domains $D_1,\ D_2$, and $D_3$ containing $z^*$ and analytic functions $t_1(z),\ t_2(z)$, and $t_3(z)$ such that 
\begin{align*}
    D(t_1(z),z)&=0\ \forall z\in D_1, \\
    D(t_2(z),z)&=0\ \forall z\in D_2, \\
    D(t_3(z),z)&=0\ \forall z\in D_3,
\end{align*}
where $D(t,z)$ is the bivariate denominator of the generating function of $P_n(z)$. Let $\displaystyle D=\bigcap_{k=1}^3 D_k$. We supress the parameter $z$ and note that by continuity, $t_1\neq t_2\neq t_3\neq0\ \forall z\in D$. From partial fractions (see (2.4), (2.5) in \cite{Zumba}) we write $P_n(z)$ as
\[
-\frac{1}{(t_1-t_2)(t_1-t_3)t_1^{n+1}} 
-\frac{1}{(t_2-t_1)(t_2-t_3)t_2^{n+1}} 
-\frac{1}{(t_3-t_1)(t_3-t_2)t_3^{n+1}}.
\]
Thus $P_n(z)$ is of the form (\ref{fn}) where
\begin{align*}
\alpha_1(z)=-\frac{1}{(t_1-t_2)(t_1-t_3)t_1},\ \beta_1(z)&=\frac{1}{t_1}, \\ 
\alpha_2(z)=-\frac{1}{(t_2-t_1)(t_2-t_3)t_2},\ \beta_2(z)&=\frac{1}{t_1}, \\ 
\alpha_3(z)=-\frac{1}{(t_3-t_1)(t_3-t_2)t_3},\ \beta_3(z)&=\frac{1}{t_3}, 
\end{align*}
all of which are analytic on $D$ since $t_1\neq t_2\neq t_3\neq0\ \forall z\in D$. For the ``no degenerate dominance" condition of Theorem \ref{sokal}, for a fixed $\omega$ on the unit circle,  we will show that the set of $z\in D$ such that $\beta_1(z)=\omega\beta_2(z)$ has empty interior. Let $\omega\in\mathbb{C}$ where $|\omega|=1$ such that $\beta_1(z)\equiv\omega\beta_2(z)\ \forall z\in D$. Then $\omega=e^{2i\theta}$ for some fixed $\theta$ and $t_1=\omega t_2$, which means $t_1=\tau e^{-i\theta},\ t_2=\tau e^{i\theta}$, and $t_3=\tau\zeta$ for some $\tau,\zeta\in\mathbb{C}\setminus\{0\}$. By Vieta's formulas, we know 
\begin{align}
\label{1}t_1+t_2+t_3&=-\frac{1}{\alpha z}, \\
\label{2}t_1t_2+t_1t_3+t_2t_3&=\frac{1}{\alpha},\ \text{ and} \\
\label{3}t_1t_2t_3&=-\frac{1}{\alpha z}.
\end{align}
Substituting $t_1=\tau e^{-i\theta},\ t_2=\tau e^{i\theta}$, and $t_3=\tau\zeta$ into Equations (\ref{1}) and (\ref{3}), we obtain
\begin{equation}\label{z}
    z=-\frac{1}{\alpha\tau^3\zeta}
\end{equation}
and
\[
    \tau(2\cos\theta+\zeta)=\tau^3\zeta,
\]
which, since $\tau\neq0$, is equivalent to 
\begin{equation}\label{4}
    \tau^2=\frac{2\cos\theta+\zeta}{\zeta}.
\end{equation}
Substituting $t_1=\tau e^{-i\theta},\ t_2=\tau e^{i\theta}$, and $t_3=\tau\zeta$ into Equation (\ref{2}), we have 
\begin{equation}\label{6}
    \tau^2(1+2\zeta\cos\theta)=\frac{1}{\alpha}.
\end{equation}
Substituting the expression for $\tau^2$ from Equation (\ref{4}) into Equation (\ref{6}) gives
\[
\frac{2\cos\theta+\zeta}{\zeta}(1+2\zeta\cos\theta)=\frac{1}{\alpha},
\]
or equivalently,
\begin{equation}\label{7}
    2\alpha\cos\theta\zeta^2+(4\alpha\cos^2\theta+\alpha-1)\zeta+2\alpha\cos\theta=0.
\end{equation}
Since $\theta$ is fixed, the sets of $\tau$ and $\zeta$ satisfying Equations (\ref{6}) and (\ref{7}) are finite, and hence the set of $z$ satisfying Equation (\ref{z}) is also finite, and a finite set has empty interior.
If $z^*\in\mathbb{C}\setminus\mathbb{R}$ such that the zeros in $t$ of $1+z^*t+t^2+\alpha z^*t^3$ are distinct and nonzero on $D$, then by (2.4) and (2.5) in \cite{Zumba} and Theorem \ref{sokal}, $z^*\in\liminf\mathcal{Z}(P_n)$ when the two smallest (in modulus) zeros have the same modulus. This is because $|t_1(z^*)|=|t_2(z^*)|$, and the second condition of Theorem \ref{sokal} is satisfied at $z^*$ since 
\[
|\beta_1(z^*)|=|\beta_2(z^*)|\geq|\beta_3(z^*)|\ \text{ is equivalent to }\ |t_1(z^*)|=|t_2(z^*)|\leq|t_3(z^*)|.
\]
The following proposition proves that $P_n(z)$ is not hyperbolic for all large $n$ if $b<0$.
\begin{proposition}\label{b}
For the sequence $\{P_n(z)\}_{n=0}^\infty$ generated by 
\[
\sum_{n=0}^\infty P_n(z)t^n=\frac{1}{1+zt-t^2+czt^3},\ c\in\mathbb{R},
\]
the polynomials $P_n(z)$ are not hyperbolic for all large $n\in\mathbb{N}$.
\end{proposition}
\begin{proof}
Case 1: $c=0$. Then the generating relation reduces to 
\[
\sum_{n=0}^\infty P_n(z)t^n=\frac{1}{1+zt-t^2}.
\]
By \cite{Tran} the zeros of this sequence lie on the curve defined by 
\[
\I\frac{z^2}{-1}=0\ \text{ and }\ 0\leq \R\frac{z^2}{-1}\leq4.
\]
In particular, the zeros lie on the curve defined by $z^2\in\mathbb{R}$ and $-4\leq z^2\leq0$, which implies that the zeros lie on the imaginary interval $(-2i,2i)$. \\
Case 2: $c\neq0$. 
Define 
\[
D(t):=1+zt-t^2+czt^3 
\]
with zeros $t_1,\ t_2$, and $t_3$, and consider the reciprocal polynomial
\[
D^*(t):=t^3+zt^2-t+cz.
\]
Let the zeros of $D^*(t)$ be $t_1^*,\ t_2^*$, and $t_3^*$. We will show that there exists a $z^*\in\mathbb{C}\setminus\mathbb{R}$ such that $|t_1^*|=|t_2^*|\geq|t_3^*|$. \\
Let $\epsilon\in\mathbb{R}\setminus\{0\}$ and let $z^*=i\epsilon$. Note that if $z=0$ then the zeros of $D^*(t)$ are $0,1$, and $-1$. By the continuity of $D^*(t)$, it is clear that as $\epsilon\to0$,
$t_1^*\to1,\ 
t_2^*\to-1,$ and 
$t_3^*\to0$.

Note that for our choice of $z^*$, $z^*=-\overline{z^*}$. Since $t_1^*$ is a zero of $D^*(t)$, 
\[
t_1^{*3}+zt_1^{*2}-t_1^*+cz=0.
\]
Observe
\begin{align*}
\left(-\overline{t_1^*}\right)^3+z^*\left(-\overline{t_1^*}\right)^2-(-\overline{t_1^*})+cz^*&=
-\overline{t_1^*}^3-\overline{z^*}\overline{t_1^*}^2+\overline{t_1^*}-c \overline{z^*} \\
&=\overline{-t_1^{*3}-z^*t_1^{*2}+t-cz^*} \\
&=0.
\end{align*}
Hence $-\overline{t_1^*}:=t_2^*$ is also a zero of $D^*(t)$, and thus 
\[
|t_1^*|=|t_2^*|\geq|t_3^*|
\]
for sufficiently small $\epsilon$. \\
Since the zeros of $D^*(t)$ are the reciprocals of the zeros of $D(t)$, we have $|t_1|=|t_2|\leq|t_3|$.
\end{proof}
\begin{proposition}
The zeros of the sequence $\{P_n(z)\}_{n=0}^\infty$ generated by 
\[
\sum_{n=0}^\infty P_n(z)t^n=\frac{1}{1+zt+t^2+\alpha zt^3},\ \alpha\in\mathbb{R},
\]
are not all real if $\alpha>1/9$ for all large $n$.
\end{proposition}
\begin{proof}
Case 1: Suppose $\alpha>1$. Consider $z=i$. Then the denominator of the generating relation becomes 
\[
D(t):=1+it+t^2+\alpha it^3.
\]
We now show that if $t_1,\ t_2$, and $t_3$ are the zeros of $D(t)$, then $|t_1|=|t_2|\leq|t_3|$.
Note that $z=i\implies z=-\overline{z}$. 
Suppose $t_1$ is a zero of $D(t)$, so that
\[
1+it_1+t_1^2+\alpha it_1^3=0.
\]
Then
\begin{align*}
1+i(-\overline{t_1})+(-\overline{t_1})^2+\alpha i(-\overline{t_1})^3&=1-i\overline{t_1}+\overline{t_1}^2-\alpha i\overline{t_1}^3 \\
&=1+\overline{i}\overline{t_1}+\overline{t_1}^2+\alpha\overline{i}\overline{t_1}^3 \\
&=\overline{1+it_1+t_1^2+\alpha it_1^3} \\
&=0.
\end{align*}
Hence $-\overline{t_1}:=t_2$ is also a zero of $D(t)$. So we have $|t_1|=|t_2|$. It remains to show that $|t_3|\geq|t_1|$. \\
Let us make the substitution $it\to y$ so that we may work with a cubic polynomial with real coefficients. Then our function $D(t)$ becomes 
\[
P(y):=1+y-y^2-\alpha y^3.
\]
Note that the discriminant of a cubic polynomial $ax^3+bx^2+cx+d$ is 
\[
\Delta=18abcd-4b^3d+b^2c^2-4ac^3-27a^2d^2.
\]
The discriminant of $P(y)$ is then 
\[
\Delta=5+22\alpha-27\alpha^2.
\]
Since $\alpha>1$, $\Delta<0$, which implies that $P(y)$ has one real zero and two non-real complex conjugate zeros. \\
Let $y_1,\ y_2$, and $y_3$ be the zeros of $P(y)$ with 
\[
y_1=\tau e^{i\theta},\ y_2=\tau e^{-i\theta},\ \text{and}\ y_3\ \text{is real}.
\]
We now show that $|y_3|>\tau$. \\
By Vieta's formulas, we have
\begin{align*}
y_1+y_2+y_3&=-\frac{1}{\alpha}, \\
y_1y_2+y_1y_3+y_2y_3&=-\frac{1}{\alpha}, \\
y_1y_2y_3&=\frac{1}{\alpha}.
\end{align*}
The last equality gives 
\[
\tau^2y_3=\frac{1}{\alpha}\implies\tau^2=\frac{1}{\alpha y_3}.
\]
Note that $|y_1|=|y_2|=\tau$, so we want to show that $\tau^2<y_3^2$, or equivalently, $y_3>\frac{1}{\sqrt[3]{\alpha}}$.
Observe
\[
P\left(\frac{1}{\sqrt[3]{\alpha}}\right)=1+\frac{1}{\sqrt[3]{\alpha}}-\frac{1}{\sqrt[3]{\alpha^2}}-\alpha\left(\frac{1}{\alpha}\right)=\frac{1}{\sqrt[3]{\alpha}}-\frac{1}{\sqrt[3]{\alpha^2}},
\]
which is greater than $0$ since $\alpha>1$. 
Furthermore,
\begin{align*}
\lim_{y\to-\infty}P(y)&=+\infty \\
\text{and }\ \lim_{y\to\infty}P(y)&=-\infty.
\end{align*}
Since $P(y_3)=0$ and $P\left(\frac{1}{\sqrt[3]{\alpha}}\right)>0$, and $y_3$ is the only real zero of $P(y)$, we get that $y_3>\frac{1}{\sqrt[3]{\alpha}}$ by the Intermediate Value Theorem. \\
Case 2: Suppose $1/9<\alpha\leq1$. Since
\[
\lim_{\theta\to0}\Delta(\theta)=9\alpha^2-11\alpha+1<0,
\]
by the continuity of $\Delta(\theta)$ there is $\theta^*$ sufficiently close to $0$ so that $\Delta(\theta^*)<0$.
For this choice of $\theta^*$, we have $\zeta^*:=\zeta(\theta^*)\notin\mathbb{R}$ and hence $\tau^*:=\tau(\theta^*)\notin\mathbb{R}$.
It is clear that $\zeta^*$ and $\tau^*$ are not real, but we must verify that $z^*:=z(\theta^*)$ is not real. \\
Suppose by way of contradiction that $z^*\in\mathbb{R}$. 
Then 
\[
g(t):=1+z^*t+t^2+\alpha z^*t^3
\]
is a polynomial in $t$ with real coefficients. \\
By Vieta's formulas,
\begin{align*}
t_1^*&:=\tau^*e^{-i\theta}, \\
t_2^*&:=\tau^*e^{i\theta}, \\
t_3^*&:=\tau^*\zeta^*
\end{align*}
are the zeros of $g(t)$. Since $g(t)$ is a polynomial with real coefficients, we know complex zeros occur in conjugate pairs. So 
\[
|t_1^*|=|t_2^*|\implies\tau^*\in\mathbb{R},
\]
a contradiction. Hence $z^*\notin\mathbb{R}$. Since $|t_3^*|>1$ by Lemma \ref{one}, we have
\[
|t_1^*|=|t_2^*|\leq|t_3^*|,
\]
and hence we have by Theorem \ref{sokal} that $P_n(z)$ is not hyperbolic for large $n$.
\end{proof}
Thus, we have shown that if $b<0$ or $\alpha>1/9$, $P_n(z)$ is not hyperbolic for all large $n$.


\section{OPEN PROBLEM}
One may wish to characterize polynomials $A(z),\ B(z),\ C(z)$ so that the sequence of polynomials $\{P_n(z)\}_{n=0}^\infty$ defined by 
\[
\sum_{n=0}^\infty P_n(z)t^n=\frac{1}{1+A(z)t+B(z)t^2+C(z)t^3}
\]
is hyperbolic. In this paper, the main result, Theorem \ref{main2}, characterizes one of the cases where $A(z)$, $B(z)$, and $C(z)$ are linear.

\clearpage



\begin{thebibliography}{99}
\addcontentsline{toc}{section}{REFERENCES} 
\singlespacing 
\thispagestyle{empty} 

\bibitem{Orthogonal} Abramowitz, Milton; Stegun, Irene A.
Handbook of mathematical functions with formulas, graphs, and mathematical tables. 
National Bureau of Standards Applied Mathematics Series, 55 For sale by the Superintendent of Documents, U.S. Government Printing Office, Washington, D.C. (1964).

\bibitem{Ch} M. Charalambides, G. Csordas, The distribution of zeros of a class of Jacobi polynomials, Proc.
Amer. Math. Soc. 138 (2010), no. 12, 4345-4357.

\bibitem{Boor} Conte, S. D.
Elementary numerical analysis: An algorithmic approach. McGraw-Hill Book Co., New York-Toronto, Ont.-London (1965).

\bibitem{Conway} Conway, John B. Functions of one complex variable. Second edition. Graduate Texts in Mathematics, 11. Springer-Verlag, New York-Berlin, (1978).

\bibitem{Eg} E\u gecio\u glu, \"Omer; Redmond, Timothy; Ryavec, Charles. From a polynomial Riemann hypothesis to alternating sign matrices. The Electronic Journal of Combinatorics (2001).

\bibitem{Forgacs} Forg\'acs, Tam\'as; Tran, Khang. Zeros of polynomials generated by a rational function with a hyperbolic-type denominator. Constr. Approx. 46 (2017), no. 3, 617-643.

\bibitem{Goh} Goh, W., He, M.X. \& Ricci, P.E. On the universal zero attractor of the Tribonacci-related polynomials. Calcolo (2009) 46: 95

\bibitem{He} M. X. He, E. B. Saff, The zeros of Faber polynomials for an m-cusped hypocycloid, J. Approx.
Theory 78 (1994), no. 3, 410-432.

\bibitem{Implicit} H\"ormander, Lars. An introduction to complex analysis in several variables. Third edition. North-Holland Mathematical Library, 7. North-Holland Publishing Co., Amsterdam, (1990). 

\bibitem{Orthogonal1} Milovanović, Gradimir V.(YU-NISEE)
Orthogonal polynomial systems and some applications. Inner product spaces and applications, 115-182, 
Pitman Res. Notes Math. Ser., 376, Longman, Harlow, (1997).

\bibitem{Munkres} Munkres, James R. Topology. Second edition of [MR0464128]. Prentice Hall, Inc., Upper Saddle River, NJ, (2000).

\bibitem{Si} Silva, A. P. da; Sri Ranga, A. Polynomials generated by a three term recurrence relation: bounds for complex zeros. Linear Algebra Appl. 397 (2005), 299-324.

\bibitem{Sokal} Sokal, Alan D.(1-NY-P)
Chromatic roots are dense in the whole complex plane. (English summary) 
Combin. Probab. Comput. 13 (2004), no. 2, 221-261.

\bibitem{Tran1} Tran, Khang. Connections between discriminants and the root distribution of polynomials with rational generating function. J. Math. Anal. Appl. 410 (2014), no. 1, 330-340. 

\bibitem{Tran} Tran, Khang. The root distribution of polynomials with a three-term recurrence. J. Math. Anal. Appl. 421 (2015), no. 1, 878-892.

\bibitem{Zumba} Tran, Khang; Zumba, Andres. Zeros of polynomials with four-term recurrence. Involve 11 (2018), no. 3, 501-518.


\end{thebibliography}
\end{document}